\documentclass{amsart}
\usepackage{color}

\newtheorem{theorem}{Theorem}[section]

\theoremstyle{definition}

\newtheorem{pro}[theorem]{Proposition}
\newtheorem{remark}[theorem]{Remark}

\numberwithin{equation}{section}

\begin{document}

\title{On Parametric Spaces of Bicentric Quadrilaterals}

\author{F. Izadi, F. Khoshnam, A. J. MacLeod \and A. S. Zargar}

\address{Farzali Izadi: Department of Pure Mathematics, Faculty of Science, Urmia University, Urmia
165-57153, Iran.}
\email{: f.izadi@urmia.ac.ir}
\address{Foad Khoshnam: Department of Pure Mathematics, Faculty of Science, Azarbaijan Shahid Madani University, Tabriz 53751-71379, Iran.}
\email{khoshnam@azaruniv.edu}
\address{Allan J. MacLeod: Mathematics and Statistics Group, University of the West of Scotland, High St., Paisley, Scotland, PA1 2BE.}
\email{allan.macleod@uws.ac.uk}
\address{Arman Shamsi Zargar: Department of Pure Mathematics, Faculty of Science, Azarbaijan Shahid Madani University, Tabriz 53751-71379, Iran.}
\email{shzargar.arman@azaruniv.edu}


\begin{abstract}
In Euclidean geometry, a bicentric quadrilateral is a convex quadrilateral that has both a circumcircle passing through the four vertices and an incircle having the four sides as tangents. Consider a bicentric quadrilateral with rational sides. We discuss the problem of finding such quadrilaterals where the ratio of the radii of the circumcircle and incircle is rational. We show that this problem can be formulated in terms of a family of elliptic curves given by $E_a:y^2=x^3+(a^4-4a^3-2a^2-4a+1)x^2+16a^4x$ which have, in general,
\(\mathbb Z/8\mathbb Z\), and in rare cases \(\mathbb Z/2\mathbb Z\times\mathbb Z/8\mathbb Z\) as torsion subgroups. We show the existence of infinitely many elliptic curves $E_a$ of rank at least two with torsion subgroup $\mathbb Z/8\mathbb Z$, parameterized by the points of an elliptic curve of rank at least one, and give five particular examples of rank $5$. We, also, show the existence of a subfamily of $E_a$ whose torsion subgroup is $\mathbb Z/2\mathbb Z\times\mathbb Z/8\mathbb Z$.
\end{abstract}

\maketitle
\section{Introduction and primary results}
In Euclidean geometry, a bicentric quadrilateral is a convex quadrilateral that has both a circumcircle passing through the four vertices and an incircle having the four sides as tangents.
In this work, we consider a bicentric quadrilateral with rational sides, and discuss the problem of finding such quadrilaterals where the ratio of the radii of the circumcircle and incircle is rational. The radii of these circles are denoted by $R$ and $r$ respectively. Examples of bicentric quadrilaterals  are squares, right kites, and isosceles tangential trapezoids.

First, we briefly recall some basic facts concerning these objects.
Let $ABCD$ be a convex quadrilateral with sides of rational length $a, b, c, d$ in anti-clockwise order. It is bicentric if and only if the
opposite sides satisfy Pitot's theorem and the opposite angles are supplementary, that is, $s=a+c=b+d$, $\angle A+\angle C=\angle B+\angle D=\pi,$ where $s$ is the semi-perimeter. 

In addition, we have
$$R=\frac{1}{4}\sqrt{\frac{(ab+cd)(ac+bd)(ad+bc)}{abcd}}$$
and $r=K/s$, where $K=\sqrt{abcd}$ is the area of the quadrilateral, see \cite{Gup}.

Define
\begin{equation}\label{1}
N:=\frac{R}{r}=\frac{s}{4abcd}\sqrt{(ab+cd)(ac+bd)(ad+bc)}
\end{equation}
and look for $N \in \mathbb{Q}^+$. From the T\'{o}th inequality, we have $N\geq\sqrt{2}$ (see \cite{Rad}). The equality holds only when the quadrilateral is a square. This shows that there is no square with rational $N$. Using the above formulae, it is straightforward to check that there is no rectangle, rhombus, or kite with this condition.
This motivates us to search for other quadrilaterals with rational values of $N$.
\begin{pro}
There are infinitely many isosceles trapezoids with $N \in \mathbb{Q}$.
\end{pro}
\begin{proof}
Assuming $a=c$ gives $d=2b-a$ and \eqref{1} can be written
\begin{equation*}
N=\frac{a\sqrt{a^2+2ab-b^2}}{b(2a-b)}
\end{equation*}
so that $N \in \mathbb{Q}$ if and only if $a^2+2ab-b^2=\Box$.

Defining $x=b/a$ means that we look for rational points on the quadric
\begin{equation*}
y^2=-x^2+2x+1.
\end{equation*}

There is a clear rational point at $(0,1)$, and the line $y=1+kx$ meets the quadric again at
\begin{equation*}
x=\frac{2-2k}{k^2+1}.
\end{equation*}

Taking $a=k^2+1$, $b=2-2k$ gives $c=k^2+1$, $d=2k^2+2k$ and $s=2k^2+2$, so assume $0 < k <1$ which gives strictly positive $a,b,c,d$, and
\begin{equation*}
N=\frac{(k^2+1)(k^2-2k-1)}{4k(k^2-1)}
\end{equation*}
if we take the correct square-root.

It is straightforward to show that this gives $\sqrt{2} \le N < \infty$ for $0 < k < 1$.
\end{proof}

This paper is organized as follows.
In Section 2, we show that a given bicentric quadrilateral, with $N$ rational, leads to a member of a family of elliptic curves, denoted by $E_a$, dependent on a single side.
We also discuss the torsion subgroup of $E_a(\mathbb Q)$, which is either $\mathbb Z/8\mathbb Z$ or $\mathbb Z/2\mathbb Z\times\mathbb Z/8\mathbb Z$.

In Section 3, we give several subfamilies of $E_a$ having strictly positive rank, and show that there exist infinitely many elliptic curves with rank at least two, parameterized by the points of an elliptic curve with rank at least one. Finally, we exhibit five examples of elliptic curves with rank $5$ having torsion subgroup $\mathbb Z/8\mathbb Z$.

In Section 4, we provide a complete parametrization of those $a$ for which $E_a$ has torsion $\mathbb Z/2\mathbb Z\times\mathbb Z/8\mathbb Z$. We give,
explicitly, $26$ known rank-three curves which have this parametrization (Table \ref{T4}, \cite{Duj}).

These results are of particular interest due to the recent appearance of curves with predetermined torsion
subgroup $\mathbb Z/8\mathbb Z$ and $\mathbb Z/2\mathbb Z\times\mathbb Z/8\mathbb Z$ in a wide variety of investigations.
See \cite{K-S,Kul,Lec1,Lec2,D-P} given in the chronological order of discovery.

\section{Elliptic curves arising from bicentric quadrilaterals}
From \eqref{1}, if $N \in \mathbb{Q}$, we must have rational $t$ such that
\begin{equation*}
t^2=(ab+cd)(ac+bd)(ad+bc).
\end{equation*}

Since scaling all the edges by the same factor does not affect $R/r$, we can assume, without loss of generality, that $d=1$,
which gives $b=s-1$ and $c=s-a$, so
\begin{align}
t^2&=(a+1)^2s^4-2(a+1)^3s^3+(a^4+8a^3+10a^2+8a+1)s^2 \label{q4} \\
&\quad  -4a(a+1)(a^2+a+1)s+4a^2(a^2+1). \nonumber
\end{align}

Defining $t=y/(a+1)$ and $s=x/(a+1)$, gives the quartic in a standard form
\begin{align*}
y^2&=x^4-2(a+1)^2x^3+(a^4+8a^3+10a^2+8a+1)x^2 \\
&\quad -4a(a+1)^2(a^2+a+1)x+4a^2(a+1)^2(a^2+1).
\end{align*}

In \cite{Mord}, Mordell shows that such quartics are birationally equivalent to an elliptic curve, and gives an algorithm for determining the elliptic curve from such a quartic. Applying this algorithm shows that the corresponding curve is
\begin{equation}\label{ec}
E_a: \, \, v^2=u^3+(a^4-4a^3-2a^2-4a+1)u^2+16a^4u
\end{equation}
which has discriminant
\begin{equation*}
\Delta=4096a^8(a+1)^2(a-1)^4(a^2-6a+1).
\end{equation*}

The transformations from the elliptic curve to the quartic \eqref{q4} are
\begin{equation}\label{sf}
s=\frac{u(a+1)^2+v}{2u(a+1)}
\end{equation}
and
\begin{equation*}
t=\frac{2a^3(s-2)-2a^2(s^2-3s+2)-2a(2s^2-3s+2)-2s^2+2s+u}{2(a+1)}.
\end{equation*}

The reverse transformations are
\begin{equation}\label{uf}
u=-2(a^3(s-2)-a^2(s^2-3s+2)-a(2s^2-3s+2+t)-s^2+s-t)
\end{equation}
and
\begin{align*}
v&=2(a+1)(2s^3(a+1)^2-3s^2(a+1)^3+2st(a+1) \\
&\quad +s(a^4+8a^3+10a^2+8a+1)-t(a+1)^2-2a(a+1)(a^2+a+1)).
\end{align*}

As an example, a simple search finds the rational quadrilateral with integer sides $\{a,b,c,d\}=\{21,28,12,5\}$ which gives $N=99/40$.
Scaling so that $d=1$ gives $a=21/5$ and $s=33/5$, which gives $t=10584/125$. The elliptic curve $E_a$ is
\begin{equation*}
v^2=u^3-\frac{22664}{625}u^2+\frac{3111696}{625}u.
\end{equation*}

Formula \eqref{uf} gives $u=1764$ and, substituting into the elliptic curve, we have $v= \pm 366912/5$. The elliptic curve has a torsion point
$(84/5, 34944/125)$, and adding this to $(1764,366912/5)$ gives $u=756/125$ and $v=532224/3125$. Using \eqref{sf}, we find $s=69/13$ and, after scaling,
the quadrilateral with $\{a,b,c,d\}=\{273,280,72,65\}$.

Note that the reverse process does not always give a real-life quadrilateral. The above curve also has a torsion point at $(1764,451584/625)$. Adding this to
$(1764,366912/5)$ gives $u=9604/225$ and $v=7990528/16875$. This gives $s=367/135$, but $c=-40/27$.

We remark that, in the spirit of MacLeod's work in \cite{Mac}, it is natural to seek integer $N$'s. But, preliminary computer searches did not yield any integer values for $N$.

The elliptic curve $E_a$ clearly has at least one torsion point of order 2 at $(0,0)$. Numerical tests suggest that the torsion subgroup is isomorphic, in general, to $\mathbb Z/8\mathbb Z$, and, in rare cases to
$\mathbb Z/2\mathbb Z\times\mathbb Z/8\mathbb Z$.

There are points of order 4 at $(4a^2,\pm 4a^2(a-1)^2)$, and order 8 at $(4a,\pm 4a(a^2-1))$ and $(4a^3,\pm 4a^3(a^2-1))$.

The above birational transformations, together with the torsion points, lead to

\begin{theorem}\label{ps}
All rational bicentric quadrilaterals with rational $N$ correspond to rational points on $E_a$ for some $a>0$, where $E_a$ will have strictly positive rank.
\end{theorem}

As we saw above, not all rational points on $E_a$ lead to real-life quadrilaterals with all sides strictly positive.
We can assume, without loss of generality, that $a>1$, so we need $s>a$ to ensure $c>0$.

From \eqref{sf}, we thus require
\begin{equation*}
\frac{u(a+1)^2+v}{2u(a+1)}>a  \hspace{1cm} \mbox{or} \hspace{1cm} \frac{v-(a^2-1)u}{2(a+1)u} > 0.
\end{equation*}

The line $v=(a^2-1)u$ passes through the origin and two of the points of order $8$, $(4a, 4a(a^2-1))$ and $(4a^3,4a^3(a^2-1))$, and these are the only
intersections. Thus, from the geometry, we have $s>a$ if $0 < u < 4a$ and $u>4a^3$, and $v>0$. Also, if $a>2\sqrt{2}+3$, the curves $E_a$ will have two
components. The line will lie below the finite component, so any point on this component will lead to $s<a$ since $u<0$.

The equation \eqref{ec} factors as $v^2=u(u-u_1)(u-u_2)$ over $\mathbb Q$,
where
\begin{align*}
u_1&=-\frac{1}{2}a^4+2a^3+a^2+2a-\frac{1}{2}+\frac{1}{2}(a+1)(a-1)^2\sqrt{a^2-6a+1}, \\
u_2&=-\frac{1}{2}a^4+2a^3+a^2+2a-\frac{1}{2}-\frac{1}{2}(a+1)(a-1)^2\sqrt{a^2-6a+1},
\end{align*}
showing that the extra points of order 2 occur when $a^2-6a+1$ is square \cite[p. 77]{Was}, for instance one can take $a=6$.

Numerical tests show that $E_a$ often has rank 0. The first rank-one curve, with $a \in \mathbb{Z}$, is obtained for $a=10$, namely $y^2=x^3+5761x^2+160000x$, with generator $G=(-32,-864)$, according to {\tt{SAGE}} \cite{Sag}. None of the points $G\pm T,$ where $T$ is a torsion point give rise to an appropriate quadrilateral. The point $2G=(8464,-1010160)$ and if we take $u=8464, v=1010160$ we have the solution $s=2764/253$, which after scaling leads to $a=2530$, $b=2511$, $c=234$ and $d=253$, and $N=21437584/3753945$.

\section{Ranks of $ E_a$ with torsion subgroup $\mathbb Z/8\mathbb Z$}
As mentioned earlier, elliptic curves with torsion subgroup $\mathbb Z/8\mathbb Z$ have appeared in a wide variety of contexts in recent years.
Writing the curve in the form $y^2=x^3+Ax^2+Bx$ is a convenient way to seek for candidates for rational points of infinite order.
Basically, the $x$-coordinates should be either divisors of $B$ or rational squares times these divisors.
In this way, we describe eight rank-one subfamilies of \eqref{ec}, along with the existence of infinitely many elliptic curves of rank at least two. Then, from the subfamilies of rank at least one, we derive five curves of rank $5$.

\subsection{Existence of positive rank elliptic curves.}
Our computations show that there are many elliptic curves $E_a$ of positive rank where their torsion  subgroups are $\mathbb Z/8\mathbb Z$.
For instance, if $x_1=-a^3$ was to be the $x$-coordinate of a point $P$ on $E_a$, we would need to find a rational $b$ such that
$$E:\, b^2=a^4-5a^3-2a^2-20a+1.$$

There is a rational point $(0,1)$ on $E$, and hence $E$ is equivalent to an elliptic curve. Using standard methods of transformation, $E$ is isomorphic to
\begin{equation*}
E':\, q^2=p^3+7668p+361881
\end{equation*}
with the transformation
\begin{equation*}
a=\frac{15p+2q+1170}{12p-819}.
\end{equation*}
The curve $E'$ has rank $2$ with generators $(-38,128)$ and $(-24,405)$ and torsion subgroup $\mathbb Z/3\mathbb Z$, with finite torsion points
$(12, \pm 675)$. There are thus an infinite number of rational points on $E'$ and hence an infinite number of values of $a$ giving $E_a$ of rank at least one.

For example, $(-24,405)$ gives $a=-60/41$ and $b= \pm 11431/1681$.

In a similar way, one readily can achieve similar results by considering the points with $x$-coordinates such as
$ka^{\ell}$ for $(k,\ell)=(2,0)$, $(8,0)$, $(1,2)$, $(2,2)$, $(8,2)$, $(2,4)$, $(8,4)$, $(\pm 1,1)$, $(\pm 2,1)$, $(\pm 8,1)$, $(-16,1)$, $(\pm 16,2)$, $(-1,3)$, $(\pm2,3)$, $(\pm 8,3)$, $(\pm16,3)$.
\subsection{Subfamilies with rank $\geq1$.}
We used computer programs to hunt for points on \eqref{ec} which will give a rational point subject to a quadratic being a square. These can be parameterized in the usual way. The points and parameterizations are exhibited in Table \ref{T3}.
\begin{table}[h]
\begin{center}
\begin{tabular}{ccr}
\hline
Subfamily~Number & $x$ & Parametrization~for~$a$ \\ \hline
1 & 4 & $(k^2-8k+11)/(k^2-5)$ \\
2 & $8a-4$ & $(k^2+12)/(2k^2-8)$\\
3 & $a(a+1)^2$ & $-(2k-3)/(k^2-1)$\\
4 & $2a(a^2+1)$ & $-2k/(k^2-1)$ \\
5 & $a^2+4a-1$ & $(k^2-4k+5)/(k^2-1)$\\
6 & $-4a^3(a-2)$ & $-4(k-1)/(k^2+3)$\\
7 & $2(a^2+2a-1)$ & $-(k^2+1)/(2k-2)$\\
8 & $a^2(-a^2+4a+1)$ & $-(2k-4)/(k^2+1)$\\
\hline
\end{tabular}
\vspace{.2cm}
\caption{Rank-one subfamilies with torsion subgroup $\mathbb Z/8\mathbb Z$}\label{T3}
\end{center}
\end{table}

We give only the details of the first rank-one subfamily, and leave the details of the remainder to the reader.

Let us consider the number 4 to be the $x$-coordinate of a point $P$ on $E_a$. This leads us to seek a rational $b$ such that $5a^2+6a+5=b^2$. Using $(a,b)=(1,4)$, we achieve the parametric solution
\begin{equation*}
a=\frac{k^2-8k+11}{k^2-5}.
\end{equation*}
For these values of $a$, the corresponding curves can be written as
$$E_k:y^2=x^3+A(k)x^2+B(k)x,$$
where
\begin{align*}
A(k)&=-2(k^8-16k^7+76k^6-16k^5-1226k^4+5456k^3-11348k^2+11984k \\
    & \qquad \quad-5167),\\
B(k)&=(k^2-8k+11)^4(k^2-5)^4,
\end{align*}
once we clear the denominators.

There is no rational value of $k$ which makes $B(k)=0$, so the only problems come from repeated roots for the cubic.
We have $A^2(k)-4B(k)$ equals
\begin{equation*}
-4096(k-1)^2(k-2)^4(k-3)^2(k^2+2k-7)(k^2-10k+17)
\end{equation*}
so that we get a non-singular elliptic curve for rational $k\neq1,2,3$.

The elliptic curve $E_k$ contains the point
$$P=\left((k^2-5)^4, 16(k-2)(k^2-4k+5)(k^2-5)^4\right).$$
By Mazur's theorem \cite[Theorem 8.11]{Was}, one can easily check that $P$ is not a torsion point.
Alternatively, it can be shown using a specialization argument, since the specialization map is a monomorphism.
For $k=0$, the elliptic curve $E_k$ turns into
$$E_0:y^2=x^3+10334x^2+9150625x$$
with
$$P=(625, 100000).$$
The height of this point is $2.34275900093414$ showing that the point has infinite order. Hence, by the specialization theorem of Silverman \cite[Theorem 20.3]{Sil}, rank~$E_k(\mathbb Q)\geq 1$.

Now, we show that torsion subgroup of $E_k$, $\mathcal T$, is generically isomorphic to $\mathbb Z/8\mathbb Z$.

Clearly, the curve $E_k$ has 2-torsion point at $T_1=(0,0)$. There will be other points of order $2$ if $A^2(k)-4B(k)=\Box$, which reduces to
$$T^2=-(k^2+2k-7)(k^2-10k+17)=-(k^4-8k^3-10k^2+104k-119)$$
having rational solutions. Cremona's RATPOINT program shows the the quartic is not locally soluble at the primes $2$ and $3$, so there is
only ever one point of order $2$.

Since the original elliptic curve $E_a$ has points of order $4$ and $8$, and we are just making rational transformations $E_k$ also must
have points of these orders. The presence of only one order $2$ point means that the torsion subgroup of $E_k$ must be $\mathbb Z/8\mathbb Z$.

The other cases are investigated in a similar manner.

\subsection{Infinitely many curves of rank $\geq2$}
We now show that there exists infinitely many elliptic curves (with torsion subgroup $\mathbb Z/8\mathbb Z$)
of rank at least two, parameterized by rank-one subfamilies.
For this, we use two methods.
The first method rests on \cite{Lec1}, while the second one is based on method of Mestre.

We begin by giving details of the first method leading to an infinitude of curves of rank $\geq2$ in the Number-1 and -5 subfamilies in Table \ref{T3}.

Consider the equation $a_1=a_5$ i.e.
$$\frac{r^2-8r+11}{r^2-5}=\frac{s^2-4s+5}{s^2-1},$$
which we show has infinitely many rational solutions. Considered as a quadratic in $s$, the resulting discriminant has to be a rational square, so
$$R^2=r^4-12r^3+30r^2-4r-31$$
for some $R \in \mathbb{Q}$. But this is equivalent to the rank-one elliptic curve
$$E:y^2-12xy-8y=x^3-6x^2+124x-744$$
being generated by $(-18,-96)$.
Now, by the specialization theorem, it is sufficient to show the independence of the points
with $x$-coordinates $Q_1=4$ and $Q_2=a^4+4a-1$. For this, we take $(x_0,y_0)=(-366/25,-9632/125)$ on $E$, then
$r=28/15$ and $a=101/341$. Then the points of infinite order
\begin{align*}
P_1&=(4, 879360/116281), \\
P_2&=(31684/116281, 1907106240/13521270961),
\end{align*}
are independent on the elliptic curve
$$y^2=x^3-\frac{6170699848}{13521270961}x^2+\frac{1664966416}{13521270961}x,$$
since the corresponding regulator has value $29.1615800873524$.

Similar results can be obtained for other combinations from Table \ref{T3}.
The following pairs of the subfamilies also lead to infinitely many elliptic curves of
rank at least two: $(i,j)=(1,8)$, (2,5), (2,7), (3,4), (3,5), (3,8), (4,5), (4,6), (4,7), (4,8), (5,7), (6,8).

By using the fifth rank-one family exhibited in Table \ref{T3}, we can also get infinitely many curves of rank at least two. The family can be written as:
$$y^2=x^3+Ax^2+Bx,$$
with
\begin{align*}
A&=-2k^8+16k^7-48k^6+48k^5+116k^4-528k^3+912k^2-784k+274,\\
B&=(k^2-1)^4(k^2-4k+5)^4.
\end{align*}
Suppose we wish a point with $x$-coordinate $-(k^2-1)^2(k^2-4k+5)^2$. This leads to the quartic $-k^4+4k^3-16k+14$ having to be a square.
This quartic is equivalent to the rank-one curve $y^2+80y=x^3-18x^2+100x-1800$, being generated by $(15,-15)$. An argument as above shows the independence.

We can do the same for the associated curve
$$\bar E_a:y^2=x^3-2Ax^2+(A^2-4B)x$$
of $E_a$, in which $A=a^4-4a^3-2a^2-4a+1$, and $B=16a^4$. The curve $\bar E_a$ is isogenous to $E_a$ via  $(x,y)\mapsto(y^2/x^2,y(B^2-x^2)/x^2)$. This map, as a group homomorphism (i.e. isogeny), carries the rational points of $\bar E_a(\mathbb Q)$ into the rational points of $E_a(\mathbb Q)$. We note that, by the same procedure, a map $\bar E_a\rightarrow\bar{\bar{E}}_a\cong E_a$ is defined. This intimate relation between $E_a$ and $\bar E_a$ makes it natural, if one is studying $E_a$, to also study $\bar E_a$.

Now, let
$$x=(1-a^2)(a^2-6a+1)$$
be the $x$-coordinate of a point on the curve $\bar E_a$. We must have $b \in \mathbb{Q}$ such that $-a^2+4a+1=b^2$. This is equivalent to choosing
$$a=-2\frac{k-2}{k^2+1}.$$
For these values of $a$, the corresponding curves $\bar E_a$ can be written as
$${\bar E}_k:y^2=x^3+A(k)x^2+B(k)x,$$
where
\begin{align*}
A(k)&=-2k^8-16k^7+40k^6-176k^5+532k^4-752k^3+360k^2+176k+ 94, \\
B(k)&=(k^4+12k^3-18k^2-4k-7)(k^2-2k+5)^2(k+3)^4(k-1)^4,
\end{align*}
and the curve ${\bar E}_k$ has the point of infinite order
\begin{align*}
P_1&=((k-1)(k+3)(k^2-2k+5)(k^4+12k^3-18k^2-4k-7),\\
& 4(k-1)(k-2)(k+3)(k^2-2k+5)(k^2-4k-1)(k^4+12k^3-18k^2-4k-7)).
\end{align*}
so long as $(k^4+12k^3-18k^2-4k-7)(k+3)(k-1)(k-2)\neq0$.

Now, considering
$$(k+3)^4(k-1)^2(k^2-2k+5)$$
as the $x$-coordinate of a new point on $\bar E_k$, requires
$$K^2=17k^4-20k^3+2k^2-12k+29$$
for a $K \in \mathbb{Q}$. Since this latter equation has a solution $(k,K)=(1,4)$, and has nonzero discriminant,
it is equivalent to the rank-one elliptic curve
$$y^2+384y=x^3+44x^2-1088x-47872.$$
Now, a similar argument as given in the start of this section based on the specialization theorem shows the existence of infinitely many curves of rank at least two.
\begin{remark}
We say that the set of points $P_i=(x_i,y_i)$ on the curve $E$, for $i=1$, 2, $\dots$, $\ell$,
forms a length $\ell$ geometric progression if the sequence $x_1$, $x_2$, $\dots$, $x_{\ell}$
forms a geometric progression.
Obviously, the curve $E_a$, for any $a$, contains the points with $x$-coordinates
$4a^i$, $1\leq i\leq 3$ forming a length-three progression.
The subfamily $E_a$ with $a=(k^2-8k+11)/(k^2-5)$ has the points with
$x_i=4a^i$, $0\leq i\leq 4$ forming a length-five geometric progression.
The subfamily $E_a$ with $a=(k^2-2k+2)/(k^2+2)$ also possesses the same property.
Notice that this latter subfamily has infinitely many $k$'s for which $ E_a$ is of rank at least one.
In fact, imposing 4 to be the $x$-coordinate of a point implies that the quartic
$K^2=4k^4-8k^3+21k^2-16k+16$
must have infinitely many solutions. But it has a solution, $(k,K)=(0,4)$ for example,
hence it is equivalent to the cubic $y^2+xy=x^3-17x-23$ whose rank is $1$.
\end{remark}

\section{Ranks of $E_a$ with torsion subgroup $\mathbb Z/2\mathbb Z\times \mathbb Z/8\mathbb Z$}
In the previous section, we studied the possible ranks of $E_a$ where the torsion subgroup was $\mathbb Z/8\mathbb Z$. In this section, we provide a complete parametrization of those $a$ for which $E_a$ has torsion subgroup $\mathbb Z/2\mathbb Z\times \mathbb Z/8\mathbb Z$. We also derive twenty-six known curves, with this torsion subgroup, which have rank $3$.
\begin{theorem}\label{4.1}
The torsion subgroup of $E_a(\mathbb Q)$ equals $\mathbb Z/2\mathbb Z\times \mathbb Z/8\mathbb Z$ if and only if $$a=-\frac{r+1}{r(r-1)}\text{ for some }r\in\mathbb Q\setminus\{0,\pm1\}.$$
\end{theorem}
\begin{proof}
Since $E_a$ has points of order $8$ at $u=4a$ and $u=4a^3$, to have $\mathbb Z/2\mathbb Z\times \mathbb Z/8\mathbb Z$ torsion only requires three points of order $2$. The points of order $8$ will still have order $8$, so $3$ rational points of order $2$ forces the torsion subgroup (over $\mathbb{Q}$)
to be $\mathbb Z/2\mathbb Z\times \mathbb Z/8\mathbb Z$, by Mazur's classification of possible torsion structures.

This occurs when
\begin{equation*}
(a^4-4a^3-2a^2-4a+1)^2-4(16a^4)=\Box=(a+1)^2(a-1)^4(a^2-6a+1).
\end{equation*}

The rational quadric $a^2-6a+1=b^2$ has the simple solution $(0,1)$, and the line $b=1+ka$ meets the curve again when
\begin{equation*}
a=\frac{2(k+3)}{1-k^2}
\end{equation*}
which is a parametrization of those $a$ having $3$ rational points of order $2$.

Substituting $k=2r-1$ gives the stated result.
\end{proof}

From this, we refer the reader to \cite{C-G} to see what is known concerning the ranks of curves with prescribed torsion subgroup $\mathbb Z/2\mathbb Z\times \mathbb Z/8\mathbb Z$.

We close this section by listing, in Table \ref{T4}, 26 out of 27 known rank-three curves which come from the family $E_a$, with $a$
defined as in Theorem \ref{4.1} (see also \cite{Duj}).

\begin{table}[h]
\begin{center}
\begin{tabular}{lll}
\hline
$r$ & Authors & Date  \\ \hline
${12}/{17}$ & Connell-Dujella &  2000 \\
${47}/{18}$ &  Dujella &  2001  \\
${133}/{86}$ & Rathbun &  2003  \\
${201}/{239}$ & Campbell-Goins & 2003 \\
${299}/{589}$, $247/160$, $281/138$, & Dujella  &  2006  \\
$281/133$ && \\
$439/17$, $569/159$ & Rathbun & 2006 \\
$923/230$ & Dujella-Rathbun & 2006 \\
$247/419$, $200/99$, $337/65$ & Flores-Jones-Rollick-Weigandt-Rathbun & 2007 \\
$1017/352$ & Dujella & 2008 \\
${999}/{76}$, ${412}/{697}$, ${349}/{230}$,   & Fisher & 2009 \\
$217/425$, $440/217$, $309/470$,   &  & \\
$496/319$, $585/391$, $219/313$,  & & \\
$336/191$ && \\
${257}/{287}$ &  Rathbun  &  2013 \\ \hline
\end{tabular}
\vspace{.2cm}
\caption{Curves with torsion subgroup $\mathbb Z/2\mathbb Z\times\mathbb Z/8\mathbb Z$ having rank $3$}\label{T4}
\end{center}
\end{table}

\section{Examples of curves with high rank}
The highest known rank of an elliptic curve over $\mathbb Q$ with torsion subgroup $\mathbb Z/8\mathbb Z$ is rank 6. The first example was found by
Elkies in 2006 and the second, recently in 2013, by Dujella, MacLeod and Peral. We refer to \cite{Duj} for the details of both curves.

The most common strategy for finding high rank elliptic curves over $\mathbb Q$ is the construction of families of elliptic
curves with  positive generic rank as high as possible, and then to search for adequate specialization with efficient
sieving tools. One popular tool is the Mestre-Nagao sum, see for example \cite{Mes, Nag}.

These sums are of the form
$$S(n,E)=\sum_{p\leq n,~p~\rm{prime}} \left(1-\frac{p-1}{\#E(\mathbb F_p)}\right) \log{p}.$$
For the subfamilies 1, 5 and 8 in Table \ref{T3}, we looked for those curves $E_a$ with $S(523,E_a)>10$ and
$S(1979, E_a)>14$, while for subfamily 4, we considered $S(523, E_a)>8$, $S(1979, E_a)>10$.
After this initial sieving, we calculated the rank of the remaining curves with {\tt{mwrank}} \cite{Cre} though we were not able to
always determine the rank exactly. The next table summarises the results found.
%======================
\begin{table}[h]
\begin{center}
\begin{tabular}{cl}
\hline
Subfamily~Number & $k$  \\  \hline
&\\
1 & $\displaystyle \frac{257}{134}, \frac{311}{129}$  \\
& \\
4 & $\displaystyle \frac{115}{28}, \frac{301}{396}, \frac{12}{233}^*$ \\
& \\
5 & $\displaystyle \frac{79}{50}^{**}$ \\
&\\
8 & $\displaystyle \frac{113}{129}^{**}$\\
& \\
\hline
\end{tabular}
\vspace{.2cm}
\caption{Rank-five curves with torsion subgroup $\mathbb Z/8\mathbb Z$}\label{T5}
\end{center}
\end{table}

For the curves with $k$ indicated by $k^*$ in the table, we have $4\leq{\rm rank}\leq5$, and, using the parity conjecture, we get (conditionally)
that the rank is equal to 5. For the family 4 with $k=389/858$, and $k=221/148$, we have respectively $1\leq{\rm rank}\leq7$, and $4\leq{\rm rank}\leq6$.
The $k$'s indicated by $k^{**}$ give rise to the same curve:
\begin{align*}
y^2+xy&=x^3-304241169811532712979315990x \\
&\quad +2065986446448965089594679105215890328100.
\end{align*}
%========================================================================================================================
\bibliographystyle{amsplain}

\end{document}